\documentclass[12pt]{article}
\usepackage{srcltx}
\usepackage{amssymb,amsmath,amsfonts,amsthm}
\usepackage{cite}

\newtheorem{lem}{Lemma}
\newtheorem*{theorem}{Theorem}
\newtheorem*{corollary}{Corollary}
\newtheorem{pr}{Proposition}

\begin{document}

\begin{center}\textbf{Thompson's conjecture for alternating group}\end{center}

\begin{center}I. B. Gorshkov \footnote{The work is supported by the 
 Russian Science Foundation (project no. 15-11-10025)}\end{center}

\medskip
\textsl{Abstract}:
Let $G$ be a finite group, and let $N(G)$ be the set of sizes of its conjugacy
classes. We show that if a finite group $G$ has trivial
center and $N(G)$ equals to $N(Alt_n)$ or $N(Sym_n)$ for $n\geq 23$, then $G$ has a composition factor isomorphic to an alternating group $Alt_k$ such that $k\leq n$ and the half-interval $(k, n]$ contains no primes. As a corollary, we prove the Thompson's conjecture for simple alternating groups.


\section{Introduction}

Consider a finite group $G$. For $g\in G$, let  
$g^G$ denote the conjugacy class of $G$ containing $g$, and $|g^G|$ denote the size of $g^G$. The centralizer of $g$ in $G$ is denoted by $C_G(g)$. Put $N(G)=\{n\  |\  \exists g\in G$ such that $|g^G|=n\}$. In
1987 Thompson posed the following conjecture concerning $N(G)$.
\medskip

\textbf{Thompson's Conjecture (see \cite{Kour}, Question 12.38)}. {\it If $L$ is a finite
simple non-abelian group, $G$ is a finite group with trivial
center, and $N(G)=N(L)$, then $G\simeq L$.}

\medskip

Thompson's conjecture was proved for many finite simple groups of Lie type.
Denote the alternating group of degree $n$ by $Alt_n$ and the symmetric group of degree $n$ by $Sym_n$. Alavi and Daneshkhah
proved that the groups $Alt_n$ with $n=p$, $n=p+1$, $n=p+2$ for prime
$p\geq5$ are characterized by $N(G)$ (see \cite{AD}). This conjecture has recently been confirmed for $Alt_{10}, Alt_{16}$,
$Alt_{22}$ and $Alt_{26}$ (see \cite{VasT}, \cite{Gor}, \cite{Xu}, \cite{Liu}). In \cite{GorA}, the
author showed that if $N(G)=N(Alt_n)$ or $N(G)=N(Sym_n)$ for $n\geq5$, then $G$ is non-solvable.
In \cite{GorA2}, we obtained some information about composition factors of a group $G$ in the case where $N(G)=N(Alt_n)$ or $N(G)=N(Sym_n)$ for $n>1361$. It was shown in \cite{GorA3} that Thompson's conjecture is valid for alternating group $Alt_n$, where $n>1361$.

Here is our main result.

\begin{theorem}
If $G$ is a finite group with trivial center such that $N(G)=N(Alt_n)$ or $N(G)=N(Sym_n)$ with $n\geq23$, then $G$ has a composition factor $S$ isomorphic to an alternating group $Alt_k$, where $k\leq n$ and the half-interval $(k, n]$ contains no primes.
\end{theorem}

Theorem and the main result of \cite{GorA3} (see Lemma \ref{Alt1361} below) imply immediately the following corollary.

\begin{corollary}
Thompson's conjecture is true for an alternating group of degree $n$ if $n\geq5$ and $n$ or $
n-1$ is a sum of two primes.
\end{corollary}

At this moment it is not known there are even positive integers, which can not be decomposed into a sum of two primes.

\section{Notation and preliminary results}
\begin{lem}[{\rm\cite{GorA3}}]\label{Alt1361}
Let $G$ be a finite group with trivial center such that $N(G)=N(Alt_n)$, where $n\geq 27$ and at least one of the numbers $n$ or $n-1$ is a sum of two primes. If $G$ contains a composition factor
$S\simeq Alt_{n-\varepsilon}$, where $\varepsilon$ is a non-negative integer such that the set $\{n-\varepsilon,...,n\}$
does not contain primes, then
$G\simeq Alt_n$.
\end{lem}

\begin{lem}[{\rm\cite[Lemma 3]{VasT}}]\label{pi}
If $G$ and $H$ are finite groups with trivial centers and $N(G) = N(H)$,
then $\pi(G) = \pi(H)$.
\end{lem}

\begin{lem}[{\rm \cite[Lemma 1.4]{GorA2}}]\label{factorKh}
Let $G$ be a finite group, $K\unlhd G$, $\overline{G}= G/K$, $x\in G$ and $\overline{x}=xK\in G/K$.
The following assertions are valid

(i) $|x^K|$ and $|\overline{x}^{\overline{G}}|$ divide $|x^G|$.

(ii) If $L$ and $M$ are neigboring members of a composition series of $G$, $L<M$, $S=M/L$, $x\in M$  and
$\widetilde{x}=xL$ is an image of $x$, then $|\widetilde{x}^S|$ divides $|x^G|$.

(iii) If $y\in G, xy=yx$, and $(|x|,|y|)=1$, then $C_G(xy)=C_G(x)\cap C_G(y)$.

(iv) If $(|x|, |K|) = 1$, then $C_{\overline{G}}(\overline{x}) = C_G(x)K/K$.

\end{lem}

\begin{lem}[{\rm \cite[Lemma 4]{VasT}}]\label{vas}
 Given a finite group $G$ with trivial center, if there exists a prime $p\in \pi(G)$ such that $p^2$ does not divide $|x^G|$ for all $x$ in $G$, then the Sylow $p$-subgroups of $G$ are elementary abelian.
\end{lem}

\begin{lem}\label{diretN}
If $G=A\times B$, then $N(G)=\{ab|a\in N(a), b\in N(b)\}$.
\end{lem}
\begin{proof} The proof is obvious.
\end{proof} 

\begin{lem}[{\rm\cite[Lemma 1.6]{GorA2}}]\label{Pord}
 Let $S$ be a non-abelian finite simple group. If $p\in \pi(S)$, then there exist $a\in N(S)$ and $g\in S$ such that
$|a|_p=|S|_p$, $|g^S|=a$ and $|\pi(g)|=1$.
\end{lem}

\begin{lem}[{\rm \cite[Lemma 14]{Vac}}]\label{vac2}
Let $S$ be a non-abelian finite simple group. Any odd element from $\pi(Out(S))$ either belongs to $\pi(P)$ or does not exceed
$m/2$, where $m=max_{p\in\pi(S)}$.
\end{lem}

\begin{pr} \label{rt1}
Let $G$ be a finite group, $\Omega$ be a set of primes such that, for all $p,q\in \Omega$ and $\alpha\in N(G)$, $p$ does not divide $q-1$, $p^2$ does not divide $\alpha$. Suppose that $g\in G$ and $|g|=t\in\Omega$. If $\rho= \pi(|g^G|)\cap\Omega\neq \varnothing$, then $G$ has a nonabelian composition factor $S$ and an element $a\in S$ such that $|a|=t$ and $\rho\subseteq\pi(|a^S|)$.
\end{pr}
\begin{proof} 
Note that, as follows from Lemma \ref{vas}, a Sylow $t$-subgroup of $G$ is elementary abelian and,
consequently, $|g^G|$ is not a multiple of $t$. Let $K$ be a maximal normal subgroup of $G$ such that the
image $\overline{g}$ of $g$ in the group $\overline{G}=G/K$ is nontrivial and $\rho\subseteq\pi(|\overline{g}^{\overline{G}}|)$. Let $R$ be a minimal normal subgroup of $\overline{G}$. It can be represented as the direct product $R=R_1\times R_2\times...\times R_l$ of $l$ isomorphic
simple groups. Order of $R$ is a multiple of some $r\in \rho\cap\{t\}$. It follows from Lemmas \ref{diretN} and \ref{Pord} that if $l>1$, then there exists $h\in R$ such that $|h^G|_r>r$. It follows from Lemma \ref{factorKh} that there exists $x\in G$ such that $|x^G|_r>r$; a contradiction. We have $l=1$ and hence $R$ is a simple group.
The rest of the proof is divided into two lemmas.

\begin{lem}\label{aa}
If $|\overline{g}^R|$ is a multiple of some $r\in\rho$, then Proposition \ref{rt1} is true.
\end{lem}
\begin{proof}
Assume that $|\overline{g}^R|$ is a multiple of some $r\in\rho$. It follows from
Lemma \ref{vac2} that $\overline{g}$ acts on $R$ as an interior automorphism. Consequently, there exists $z\in R$ such that $h^{z}=h^{\overline{g}}$ for all $h\in R$.

It follows from Lemma \ref{Pord} and the fact that $r^2$ does not divide $\alpha$ for any $r\in \Omega$ and $\alpha\in N(G)$ that $k^2$ does not divide $|R|$ for any $k\in\Omega$. Suppose that there exists $r_1\in(\pi(|z^{\overline{G}}|)\cap\Omega)\setminus\pi(|z^R|)$. Since some Sylow $r_1$-subgroup of $R$ centralizes $z$, if follows from Frattini argument that some Sylow $r_1$-subgroup $H$ of $\overline{G}$ normalizes $\langle z\rangle$. Since $|Aut(\langle z\rangle)|=t(t-1)$, the subgroup $H$ centralizes $\langle z\rangle$ and $r_1\not\in\pi(|z^{\overline{G}}|)$, a contradiction. Thus, $\pi(|z^{\overline{G}}|)\cap\rho=\pi(|z^R|)\cap\rho$.

Assume that there exists $r_2\in\rho\setminus\pi(|z^R|)$. Let us show that $\overline{G}$ contains an element $h$ of order $r_2$ such that $t\in\pi(|h^{\overline{G}}|)$ and $R\leq C_{\overline{G}}(h)$. Let $\overline{K}<\overline{G}$ be a maximal normal subgroup such that $|\widetilde{g}^{\overline{G}/\overline{K}}|$ is a multiple of $r_2$, where $\widetilde{g}\in \widetilde{G}=\overline{G}/\overline{K}$ is image of $g$. Let $\widetilde{R}<\widetilde{G}$ be a minimal normal subgroup. It is clear that $|\widetilde{R}|$ is a multiple of $r_2$ or $t$. As above, it can be shown that $\widetilde{R}$ is a simple group. Since $\widetilde{R}$ contains no outer automorphisms of order $t$ and $r_2$, we obtain that any element of orders $t$ or $r_2$ induces an inner automorphism of $\widetilde{R}$. Since $\overline{K}$ is maximal, we have that $|(\widetilde{g}\widetilde{R}/\widetilde{R})^{\widetilde{G}/\widetilde{R}}|$ is not a multiple of $r_2$. Therefore  $C_{\widetilde{G}/\widetilde{R}}(\widetilde{g}\widetilde{R}/\widetilde{R})$ contains some Sylow $r_2$-subgroup $T$ of $\widetilde{G}/\widetilde{R}$. Let $H$ be the full preimage of $\langle\widetilde{g}\widetilde{R}/\widetilde{R},T\rangle$. Hence, $H\simeq(\widetilde{R}\times T)\leftthreetimes \langle\widetilde{g}\rangle$. Since $H$ contains some Sylow $r_2$-subgroup and an element $\widetilde{g}$, we have $r_2\in\pi(|\widetilde{g}^H|$. Therefore $|\widetilde{g}^{\widetilde{R}}|$ is a multiple of $r_2$.  As above, it can be shown that there exists an element $\widetilde{h}\in\widetilde{R}$ such that $|\widetilde{h}^{\widetilde{R}}|$ is a multiple of $t$. Since $\overline{K}>R$ and any $r_2$-element acts on $R$ as an inner automorphism, the element $\widetilde{h}$ have some preimage $h\in \overline{G}$ such that $h$ acts trivially on $R$. Since $|\widetilde{h}^{\widetilde{G}}|$ divides $|h^{\overline{G}}|$, we have $t\in(\pi|h^{\overline{G}}|)$.

If follows from Lemma \ref{Pord}, that there exists $u\in R$ such that $|u|\neq r_2$ and $t\in\pi(|u^R|)$. It follows from Lemma \ref{factorKh} that $t^2$ divides $|\overline{uh}^{\overline{G}}|$, a contradiction. Thus, $\rho\in |z^R|$. Therefore $z$ is the desired element and $R=S$.
\end{proof}

\begin{lem}\label{bb}
$\pi(|\overline{g}^R|)\cap\rho\neq\varnothing$.
\end{lem}
\begin{proof}
Assume that $\pi(|\overline{g}^R|)\cap\rho=\varnothing$. Since $K$ is a maximal subgroup, there exists $r\in \rho\setminus\pi((\overline{g}R/R)^{\overline{G}/R})$. Since $\pi(|\overline{g}^R|)\cap\rho=\varnothing$, we see that $C_R(\overline{g})$ contains some Sylow $r$-subgroup $H$ of $R$. Let $N=N_{\overline{G}}(H)$, $T=N_R(H)$, $\overline{N}=N/H$, $\overline{T}=T/H$, $\overline{g'}=\overline{g}H/H$. From Frattini argument it follows that $N/T\simeq\overline{G}/R$, in particular, $r\not\in\pi(|(\overline{g}T/T)^{N/T}$. Since $N$ contains some Sylow $r$-subgroup of $\overline{G}$ and the element $\overline{g}$, we obtain $r\in\pi(|\overline{g}^{N_{\overline{G}}(H)}|)$. Since $H\leq C_N(\overline{g})$ and $(|\overline{g}|,r)=1$, we have $r\in\pi(|\overline{g'}^{\overline{N}}|)$. Let $F\in Syl_t(\overline{T})$ be such that $\overline{g'}\in C_{\overline{N}}(F)$, $L=N_{\overline{T}}(F)$, $B=N_{\overline{N}}(F)$, $\widetilde{g'}=\overline{g'}F/F$, $\widetilde{L}=L/F$, $\widetilde{B}=B/F$. From Frattini argument it follows that $B$ contains some Sylow $r$-subgroup $A$ of $\overline{N}$. Since $|R|_t\leq t$, we have $|F|\leq t$, in particular, $A\in C_{\overline{N}}(F)$. Therefore $r\in\pi|\widetilde{g'}^{\widetilde{B}}$. Since $r,t\not\in\pi(|\widetilde{L}|)$, it follows that $r\in\pi|(\widetilde{g'}\widetilde{L})^{\widetilde{B}/\widetilde{L}}$, $\widetilde{L}/\widetilde{B}\simeq \overline{G}/R$, a contradiction.
\end{proof}

The statement of the proposition follows from Lemmas \ref{aa} and \ref{bb}.
\end{proof}

\begin{lem}\label{rt2}
 Let $G, \Omega$ and $g$ be as in Proposition \ref{rt1}. If there exists $r\in\pi(|g^G|)\cap\Omega$, then $G$ contains
an element $h$ of order $r$ such that $t\in \pi(|h^G|)$.
\end{lem}
\begin{proof} By Lemma \ref{rt1}, $G$ has a composition factor $S$ such that $\overline{g}\in S, |\overline{g}|=t$ and $|\overline{g}^S|$ is a multiple of r. It follows from Lemmas \ref{Pord} and \ref{factorKh} that Sylow $t$- and $r$-subgroups of $S$ are cyclic
groups of simple orders. Let $\overline{h}\in S$ and $|\overline{h}|=r$. Assume that $|\overline{h}^S|$ is not a multiple of $t$. Then,
there exists an element $x\in S$, $|x|=t$, such that $\langle x\rangle< C_S(\overline{h})$. The subgroup $\langle x\rangle$ is a Sylow $t$-subgroup of $S$. Consequently, there exists $y\in S$ such that $(\langle x \rangle)^y=\langle\overline{g}\rangle$ and, hence, $\overline{h}^y\in C_S(\overline{g})$, a contradiction. Thus, $t\in|\overline{h}^S|$. Consequently, $G$ contains an element $h$ of order $r$ such that
$|h^G|$ is a multiple of $t$.
\end{proof}

From now let us fix $V_n\in\{Alt_n, Sym_n\}$, for $n\geq5$, and a finite group $G$
such that $N(G) = N(V_n)$. Let $\Omega=\{t|t$ is a prime, $n/2<t\leq n\}$ be an ordered set, and $p$ be the largest number of $\Omega$.

\begin{lem}[{\rm \cite[Lemma 2.3]{GorA2}}]\label{ConClass}
Suppose that $t\in \Omega$, $\alpha\in N(G)$ and $\alpha$ is not a multiple of $t$. Then $\alpha$ is $|V_n|/t|C|$ or $|V_n|/|V_{t+i}||B|$, where $C = C_{V_{n-t}}(g)$ for some $g\in V_{n-t}$, $t+i\leq n$, and $B=C_{V_{n-t-i}}(h)$
for some $h\in V_{n-t-i}$.
\end{lem}

 Put $\Phi_t=\{\alpha\in N(G)\mid  \alpha=|V_n|/(t|C|),\; C=C_{V_{n-t}}(g)\mbox{ for some }g\in V_{n-t}\}$ and 
 $\Psi_t=\{\alpha\in N(G)\mid  \alpha=|V_n|/ ( |V_{t+i}||B|) 
 \mbox{ for some }i\geq0 \mbox{ and } t+i<n-1\mbox{ where  }B=C_{V_{n-t-i}}(g) \mbox{ for some } g\in V_{n-t-i}\mbox{ such that  }g\mbox{ moves } n-t-i\mbox{ points}\}$. Observe that the definitions of the sets $\Phi$ and $\Psi$ do not assume that $t$ is a prime.

\begin{lem}[{\rm \cite[Lemma 2.4]{GorA2}}]\label{Psi}
Let $t_i\in \Omega, 1\leq i<|\Omega|$. The set $\Psi_{t_i}\setminus \Psi_{t_{i+1}}$ is empty iff $n-t_i=2$ and $V=Alt$.
\end{lem}

Given a finite set of positive integers 
$\Theta$, we define a directed graph $\Gamma(\Theta)$ with the vertex set $\Theta$  and with edges $\overrightarrow{ab}$ 
whenever $a$ divides $b$. Denote by $h(\Theta)$ the maximal length of a path in $\Gamma(\Theta)$.

\begin{lem}[{\rm \cite[Lemma 8]{GorA}}]\label{hz}
The following claims hold:
\begin{enumerate}
\item{If $n-t=2$ then $h(\Psi_p)\leq 1$; }
\item{If $n-t=3$ then $h(\Psi_p)\leq 2$; }
\item{If $n-t=4$ then $h(\Psi_p)\leq 3$; }
\item{If $n-t=5$ then $h(\Psi_p)\leq 5$; }
\item{If $n-t=6$ then $h(\Psi_p)\leq 6$; }
\item{If $n-t=7$ then $h(\Psi_p)\leq 8$; }
\item{If $n-t=8$ then $h(\Psi_p)\leq 11$; }
\item{If $n-t=9$ then $h(\Psi_p)\leq 14$; }
\item{If $n-t=10$ then $h(\Psi_p)\leq 18$; }
\item{If $n-t=11$ then $h(\Psi_p)\leq 21$; }
\item{If $n-t=12$ then $h(\Psi_p)\leq 26$; }
\item{If $n-t=13$ then $h(\Psi_p)\leq 30$; }
\item{If $n-t=18$ then $h(\Psi_p)\leq 69$. }
\end{enumerate}
\end{lem}
\begin{lem}\label{hz2}
The following claims hold:
\begin{enumerate}
\item{If $23\leq n\leq26$ then $h(\Psi_p)\leq2$;}
\item{If $31\leq n\leq36$ then $h(\Psi_p)\leq2$;}
\item{If $113\leq n\leq124$ then $h(\Psi_p)\leq11$;}
\item{If $139\leq n\leq148$ then $h(\Psi_p)\leq9$;}
\item{If $199\leq n\leq210$ then $h(\Psi_p)\leq12$;}
\item{If $211\leq n\leq222$ then $h(\Psi_p)\leq12$;}
\item{If $317\leq n\leq336$ then $h(\Psi_p)\leq17$;}
\item{If $523\leq n\leq540$ then $h(\Psi_p)\leq26$;}
\item{If $887\leq n\leq905$ then $h(\Psi_p)\leq35$;}
\item{If $1129\leq n\leq1150$ then $h(\Psi_p)\leq39$;}
\item{If $1327\leq n\leq1360$ then $h(\Psi_p)\leq58$;}
\end{enumerate}

\end{lem}
\begin{proof}
Using \cite{GAP} we obtain 
the required bounds for $h(\Psi_p)$.
\end{proof}

Let $\pi$ be some set of primes. A finite group is said to have property $D_{\pi}$ if it contains a Hall $\pi$-subgroup and all its Hall $\pi$-subgroups are conjugate. For brevity, we will write $G\in D_{\pi}$ if a group $G$ has the property $D_{\pi}$.

\begin{lem}[{\rm \cite[Corollary 6.7]{DpiB}}]\label{Dpi} Suppose that $G$ is a finite group and $\pi$ is some set of primes.
Then $G\in D_{\pi}$ if and only if each composition factor of $G$ has property $D_{\pi}$.
\end{lem}

\begin{lem} [{\rm \cite{HallExB}}]\label{HallEx} Suppose that $G$ is a finite group and $\pi$ is some set of primes. If $G$ has a
nilpotent Hall $\pi$-subgroup, then $G\in D_{\pi}$.
\end{lem}

\section{Proof of Theorem}

Take a finite group $G$ with $N(G)=N(V_n)$, where $5\geq n\geq 1361$ and $Z(G)=1$. Put $\Omega=\{t \mid 
n/2<t\leq n, \; t \mbox{ is a prime}\}$ and denote the largest number of $\Omega$ by $p$. Lemma \ref{pi} implies 
that $\pi(G)\supseteq\pi(V_n)$, in particular, $\Omega \subseteq \pi(G)$. In view of the main result of \cite{AD}, we assume that for
$V_n = Alt_n$ the numbers $n$ and $n - 1$ are not prime.

\begin{pr}\label{exist}
There exists an element $g\in G$ such that $|g|\in\Omega$ and $|g^G|\in\Phi_{|g|}$.
\end{pr}
\begin{proof}
Assume that there is no element $g\in G$ such that $|g|\in\Omega$ and $|g^G|\in\Phi_{|g|}$.
\begin{lem}\label{gPsi}
If $|g|\in \Omega$, then $|g^G|\in\Psi_{p}$.
\end{lem}
\begin{proof}
Assume that $|g^G|\in \Psi_{|g|}\setminus \Psi_{p}$. Then $|g^G|$ is a multiple of $p$. Consequently, $|g|\neq p$. By Lemma \ref{rt2}, $G$ contains an element $h$ of order $p$ such that $|g|\in\pi(|h^G|)$ and, consequently, $|h^G|\not\in\Psi_p$. It follows from Lemma \ref{ConClass} that $|h^G|\in\Phi_p$; a contradiction.
\end{proof}

\begin{lem}\label{Dthe}
$G\in D_{\Omega}$.
\end{lem}
\begin{proof}
It follows from Lemma \ref{Dpi} that it is sufficient to verify the property $D_{\Omega}$ for every composition
factor of $G$.
Let $S$ be a nonabelian composition factor of $G$ such that $|\pi(S)\cap\Omega|\geq 2$ and
let $r$ and $t$ be two different elements from $\pi(S)\cap\Omega$. By Lemma \ref{factorKh}, there are no multiples of $r^2$
or $t^2$ in $N(S)$. By Lemma \ref{Pord}, a Sylow $a$-subgroup has order $a$ for any $a\in\{r,t\}$. By Lemmas \ref{gPsi} and \ref{factorKh}, $S$ contains a Hall $\{r, t\}$-subgroup $H$ of order $rt$. By the definition of the numbers $r$ and $t$, the group $H$ is cyclic. It follows from Lemma \ref{HallEx}, that $S\in D_{\{t, r\}}$. Let $l\in\pi(S)\cap\Omega\setminus \{t,r\}, g\in S, |g|=l$. Since $|g^S|$ is not multiple of $t$ or $r$, we have $H^x<C_S(g)$ for some $x\in S$. Consequently, $S$ contains a cyclic Hall $\{t,r,l\}$-subgroup. Using Lemma \ref{HallEx}, we obtain that $S\in D_{\{t,r,l\}}$. Repeating this procedure $|\pi(S)\cap\Omega|$ times, we obtain that $S\in D_{\Omega}$.
\end{proof}

\begin{lem}\label{abelian}
Hall $\Omega$-subgroup of $G$ is abelian.
\end{lem}
\begin{proof}
As follows from Lemma \ref{vas}, Sylow $t$-subgroup of $G$ is abelian for any $t\in \Omega$. Assume that a Hall $\Omega$-subgroup of $G$ is nonabelian. Then $G$ contains a nonabelian Hall $\{r, t\}$-subgroup $H$ for some $r, t\in \Omega$. Let $R<G$  be a maximal normal subgroup such that the image $\overline{H}$ of $H$ in the group $\overline{G}=G/R$ is nonabelian. The group $\overline{H}$ contains a normal $l$-subgroup $T$, for some $l\in\{r, t\}$. Note that $g\in \overline{G}, |g|\in\{r,t\}\setminus \{l\}$, acts nontrivially on $T$. We have $T=C_T(g)\times[T,g]$, where $\langle g,[T,g]\rangle$ is a Frobenius group. Since $l-1$ is not a multiple of $|g|$, we obtain that $|[T,g]|>l$ and $T$ is a acyclic group. By the definitions of the groups $R$ and $T$, we obtain that $T$ lies in some minimal normal subgroup $K$ of the group $G$. If $K$ is solvable, then $K=T$
is an elementary abelian group and, consequently, the subgroup $K\cap C_K(g)$ is a Sylow $l$-subgroup
of $C_K(g)$. As follows from \ref{factorKh}, $G$ contains a preimage $h$ of $g$ such that $|h^G|$ is a multiple of $|[T,g]|$, a contradiction. Therefore, $K=S_1\times S_2\times...\times S_m$, where $S_i$ is a nonabelian simple group for $1\leq i\leq m$. It can be shown as in Lemma \ref{rt1} that $m=1$. As noted in Lemma \ref{Dthe}, Hall $\{r,t\}$-subgroup of any composition factor is cyclic. We obtain a contradiction with the fact
that $K$ contains an acyclic $l$-subgroup $T$.
\end{proof}

\begin{lem}\label{OOmega}
If $T$ be a Hall $\Omega$-subgroup of $G$ and $\Upsilon=\{|g^G|$, $g\in T\}$, then $|\Omega|\leq h(\Upsilon)$.
\end{lem}
\begin{proof} Let $g_1\in T$ and $|g_1|=t_1\in \Omega$. By Lemma \ref{Psi}, $G$ contains an element $r_1\in G$ such that $|r_1^G|\in \Psi_{t_1}\setminus \Psi_{t_2}$, where $t_2$ is the smallest number of $\Omega\setminus\{t_1\}$. Since $G\in D_{\Omega}$ (see Lemma \ref{Dthe}), we can assume that $r_1\in C_G(g_1)$ and a Hall $\Omega$-subgroup of $C_G(r_1)$ lies in $T$. Consequently, there exists $g_2\in T$ such that $|g_2|=t_2$ and $C_G(g_1)\neq C_G(g_2)$. By Lemma \ref{abelian}, the group $T$ is abelian. Thus, $|(g_1g_2)^G|>|g_1^G|$. Note that $|g_1^G|$ divides $|(g_1g_2)^G|$. Repeating this procedure $|\Omega|$ times, we obtain the set $\Sigma=\{g_1, g_1g_2, g_1g_2g_3,..., g_1g_2...g_{|\Omega|}\}$ such that $|g_1^G|$ divides $|(g_1g_2)^G| |...||g_1g_2...g_{|\Omega|}|$. In particular, $h(\Upsilon)\geq |\Omega|$.
\end{proof}

\medskip
\begin{lem}\label{l27125}
$n\not\in\{27,28,125,126\}$
\end{lem}
\begin{proof}
Assume that $n\in\{27,28,125,126\}$. Let $t=3$ if $n\in\{27,28\}$ and $t=5$ if $n\in\{125,126\}$, $T\in Syl_G(t)$, $g\in C_G(T)$. We have  $|g^G|=|V_n|/|V_n|_t$, in particular, $|g^G|$ is a minimal and maximal element in the set $N(G)\setminus\{1\}$. Let $K\unlhd G$ be a minimal normal subgroup such that $5\in\pi(K)$. Assume that there exists $r\in\Omega\cap\pi(G/K)$. Let $h\in G$, $|h|=r$. Without loss of generality, we can assume that $h\in N_G(Z(T\cap K)$ and $T\cap C_G(h)\in Syl_{N_G(h)}(t)$. Since any multiple $a\in N(G)$ of $|h^G|$ is not a multiple of $|g'^G|$ for any  $g'\in Z(T)$, the element $h$ acts without fixed points on $Z(T\cap K)$. Therefore $|h^G|_t\geq t^x$, where $x$ is such that $|h|$ divides $p^x-1$. But $p^x>|\alpha|_p$ for any $\alpha\in Psi_{p}$; a contradiction. Let $R\lhd K$ be a maximal normal subgroup, $\Upsilon=\{17,23\}$ if $n\in\{27,28\}$ and $\Upsilon=\{113,109\}$ if $n\in\{125,126\}$. Using Frattini argument we get that the set $\pi|R|\cap\Upsilon$ is empty. Hence, $\Upsilon\subseteq \pi(K/R)$. Since $R$ is maximal, $K/R$ is simple. It follows from \cite{Zav}, that $K/R\in\{Alt_{23}, Alt_{24}..., Alt_{28}, Fi_{23}, Alt_{113},...,Alt_{126}, U_7(4)\}$. Since $G\in D_{\Omega}$, we have $K/R\in D_{\Omega}$. Let $O$ be a Hall $\Upsilon$-subgroup of $K/R$. It follows from Lemma \ref{abelian}, that $O$ is abelian, a contradiction.
\end{proof}

It follows from Lemma \ref{l27125}, that $n\not\in\{27,28,125,126\}$. We obtain From Lemmas \ref{hz} and \ref{hz2}, that $|\Omega|>h(\Psi_{p})$, a contradiction with Lemma \ref{OOmega}. This completes the proof of Proposition \ref{exist}.

\end{proof}

It follows from Proposition \ref{exist}, that there exists an element $g\in G$ such that $|g|\in\Omega$ and $|g^G|\in\Phi_{|g|}$. It follows from Lemma \ref{rt1} that there exists a composition factor $S$ of $G$ and $\overline{g}\in S$ such that $|\overline{g}|=|g|, \pi(|\overline{g}^S|)\cap\Omega=\Omega\setminus\{|g|\}$ and $|S|$ is not a multiple of $r^2$ for any $r\in \Omega$.

\begin{lem}\label{Lie}
If $n\geq23$, then $S$ is not isomorphic to a finite simple group of Lie type.
\end{lem}
\begin{proof} As noted above, $\Omega\subset\pi(S)$. Let $\Lambda(q^m)$ be an exceptional group of Lie type over field of order $q^m$, where $q<663$ is a prime, $m\geq 1327$. Using \cite{GAP} we can obtain that $\pi(\Lambda(q^m))$ either contains a number greater than $1327$ or does not contains  $\Omega(r)$, where $\Omega(r)=\{t \mid 
r/2<t< r, \; t \mbox{ is a prime}\}$ for $r$ greater or equal than the maximal number of $\pi(\Lambda(q^m))$. Since $\pi(\Lambda(q^m))$ contains a primitive divisor of $q^m-1$, we obtain that if $m>1327$, then $\pi(\Lambda(q^m))$ contains a number greater than $1327$. $|\Lambda(q^m)|$ is a multiple of $q^2$. Therefore $S$ is not isomorphic to an exceptional group of Lie type. Let $L=\Lambda_k(q^m)$ be a finite simple classical group of Lie type of Lie rank $k$ over the field $q^m$. The set $\pi(L)$ contains a primitive divisor of $q^{mk}-1$ or a primitive divisors of $q^{2mk}-1$. Hence, if $mk>1327$, then $\pi(L)$ contains a number greater than $1327$. If $k>1$, then $|L|$ is a multiple of $q^2$. Using \cite{GAP} we get that $\pi(L)$ contains a number greater than $1327$ or does not contain $\Omega(r)$ for $r$ greater or equal than the maximal number of $\pi(\Lambda_k(q^m)))$.
\end{proof}

\begin{lem}\label{Spar}
$S$ not isomorphic to a finite simple sparadic group.
\end{lem}
\begin{proof}
 The proof is obvious.
\end{proof}

It follows from Lemmas \ref{Lie} and \ref{Spar}, that $S\simeq Alt_m$. Since $p\in\pi(S)$, we obtain the required equality $m\geq p$. Thus, Theorem is proved. Corollary follows from Theorem and Lemma \ref{Alt1361}.

\end{document}